\documentclass[nocrop, oneside]{ipart_v1}
\usepackage[T1]{fontenc}
\usepackage{tikz, tikz-cd, comment, enumerate, stmaryrd, amsmath, amsthm}
\usepackage[shortlabels]{enumitem}
\usetikzlibrary{positioning}

\newcommand{\uu}{\underline}
\newcommand{\ol}{\overline}
\newcommand{\mc}{\mathcal}
\newcommand{\wh}{\widehat}
\newcommand{\mf}{\mathfrak}
\newcommand{\ZZ}{\mathbb{Z}}

\newcommand{\PP}{\mathbb{P}}
\newcommand{\QQ}{\mathbb{Q}}

\newcommand{\FF}{\mathbb{F}}

\DeclareMathOperator{\ord}{ord}
\DeclareMathOperator{\rank}{rank}
\DeclareMathOperator{\Gal}{Gal}
\DeclareMathOperator{\Proj}{Proj}
\DeclareMathOperator{\Ext}{Ext}

\newtheorem{Theorem}{Theorem}[section]
\newtheorem{Lemma}[Theorem]{Lemma}
\newtheorem{Corollary}[Theorem]{Corollary}

\newtheorem{Definition}[Theorem]{Definition}

\newtheorem{Question}[Theorem]{Question}
\theoremstyle{remark}
\newtheorem{Remark}[Theorem]{Remark}
\numberwithin{equation}{section}
\begin{document}

\title[On $p$-degree of elliptic curves]{On $p$-degree of elliptic curves}
\author[J. Garnek]{J\k{e}drzej Garnek}

\begin{abstract}
  In this note we investigate the $p$-degree function of an elliptic curve $E/\QQ_p$. The $p$-degree measures the least complexity of a non-zero $p$-torsion point on $E$. We prove some properties of this function and compute it explicitly in some special cases.
\end{abstract}

\maketitle

\section{Introduction}
Let $p \neq 2, 3$ be a prime. In this paper we define the \textbf{$p$-degree} of an elliptic curve $E$
over the field $\QQ_p$ to be:
\begin{equation*}
  d_p(E) = \min \{[\QQ_p(P):\QQ_p]: P \in E[p], \, P \neq \mc O\},
\end{equation*}
where $\QQ_p(P)$ denotes the field obtained by adjoining to $\QQ_p$ the coordinates of a point $P \in E(\ol{\QQ}_p)$. It turns out that the $p$-degree of an elliptic curve with good reduction
depends only on the reduction $\bmod \, p^2$:
\begin{Theorem}\label{t:dp_depends_on_reduction}
  If the elliptic curves $E_1, E_2/\QQ_p$ have good reduction and their reductions to $\ZZ/p^2$ are isomorphic then $d_p(E_1) = d_p(E_2)$. 
\end{Theorem}\noindent
In particular, curves with low $p$-degree correspond to the canonical lift ${\bmod \, p^2}$ and in this case we can derive an explicit formula for the $p$-degree. Let $\ord_p x$ denote the order of $x$ in the group $(\ZZ/p)^{\times}$ and $a_p(E)$ be the trace of Frobenius endomorphism for an elliptic curve $E$ defined over a finite field.
\begin{Theorem} \label{t:implications}
  Let $E/\QQ_p$ be an elliptic curve with good reduction. Let us consider the following statements:
  \begin{enumerate}[(1)]
    \item $d_p(E) < p-1$,
    \item $E_{\FF_p}$ is ordinary and $E_{\ZZ/p^2}$ is a canonical lift of $E_{\FF_p}$,
    \item $E(\QQ_p^{un})[p] \neq 0$, where $\QQ_p^{un}$ is the maximal unramified extension of $\QQ_p$,
    \item $E_{\FF_p}$ is ordinary and $d_p(E) = \ord_p a_p(E)$.
  \end{enumerate}
Then $(1)$ implies $(2)$, $(2)$ and $(3)$ are equivalent, and $(3)$ implies $(4)$.
\end{Theorem}
Failures of complementary implications are discussed in Remark~\ref{r:other_implications}.
A~result similar to Theorem \ref{t:implications} was partially stated already in~\cite{WestonDavid}. Our proof of Theorem~\ref{t:implications} uses classification of elliptic curves over finite rings by the j-invariant and an~effective version of Serre-Tate theorem, which we prove.

Investigating the $p$-degree is especially interesting when $E$ is a fixed curve over the field of rational numbers $\QQ$ and $p$ varies over primes. In particular it is natural to ask about the asymptotic behaviour of the $p$-degree:
\begin{Question}\label{q:asymptotic_behaviour}
  Does the $p$-degree of a fixed elliptic curve tend to infinity as $p$ becomes large?
\end{Question}
The authors of \cite{WestonDavid} predict that for an elliptic curve without complex multiplication
the answer to Question \ref{q:asymptotic_behaviour} is affirmative. They justify this conjecture by a simple heuristics and an averaging result. The heuristics given by David and Weston doesn't work in the case of elliptic curves with CM. We try to verify Question \ref{q:asymptotic_behaviour} in this case. The first ingredient is Theorem~\ref{t:implications}, which yields an explicit formula in the case when $p$ is ordinary. For the remaining primes we apply the following result:
\begin{Theorem} \label{t:supersingular}
  Let $E/\QQ_p$ be an elliptic curve with good supersingular reduction. Then $d_p(E) = p^2 - 1$.
\end{Theorem}
Combining methods mentioned above we get explicit formulas for the $p$\nobreakdash-degree of elliptic curves with complex multiplication by Gauss and by Eisenstein integers. Let $E_{A, B}$ denote the elliptic curve given by the Weierstrass equation $y^2 = x^3 + Ax + B$.
\begin{Theorem} \label{t:explicit_cm}
  Let $D \in \ZZ$, $D \neq 0$ and assume that $p \nmid 6 D$. Then:
  \begin{equation*}
    d_p(E_{D, 0}) =
    \begin{cases}
      \ord_p \left( (-D)^{\frac{p-1}{4}} \cdot (2s) \right), & \text{ for } p \equiv 1 \pmod 4,\\
      p^2 - 1, & \text{ for } p \equiv 3 \pmod 4,
    \end{cases}
  \end{equation*}
  \begin{equation*}
    d_p(E_{0, D}) =
    \begin{cases}       
      \ord_p \left( - (4D)^{\frac{p-1}{6}} \cdot A \right), & \text{ for } p \equiv 1 \pmod 3,\\
      p^2 - 1, & \text{ for } p \equiv 2 \pmod 3,
    \end{cases}
  \end{equation*} 
  where:
  \begin{itemize}
    \item $s$ is defined for $p \equiv 1 \pmod 4$ by the equation
      \begin{equation*}
       p = s^2 + t^2
      \end{equation*}
      and conditions $2 \nmid s$ and $s + t \equiv 1 \pmod 4$,
    \item $A$ is defined for $p \equiv 1 \pmod 3$ by the equation
      \begin{equation*}
       4p = A^2 + 3B^2
      \end{equation*}
      and conditions $A \equiv 1 \pmod 3$, $3|B$.
  \end{itemize}
\end{Theorem}
Theorem \ref{t:explicit_cm} and results of Cosgrave and Dilcher from \cite{CosgraveDilcher} allow us to show that if a certain recurrence sequence contains infinitely many primes, then Question \ref{q:asymptotic_behaviour} has negative answer for elliptic curves with complex multiplication by $\ZZ[i]$ (cf. Corollary~\ref{c:recurrence}).

Finally we compute the $p$-degree for elliptic curves with multiplicative reduction, by reducing the problem to an~investigation of the Tate curve (cf. Theorem~\ref{t:multiplicative}).\\ \\
\textbf{Outline of the paper.} Section \ref{s:preliminaries} provides a quick overview of facts related to elliptic curves over rings, inculding a classification of elliptic curves over complete rings by their $j$-invariant and an effective version of Serre-Tate theorem. In Section \ref{s:low_p_degree} we prove Theorems \ref{t:dp_depends_on_reduction} and \ref{t:implications} using the theory of formal groups applied to elliptic curves over finite rings. Finally, in the last section we investigate curves with good supersingular reduction, complex multiplication and bad multiplicative reduction. \\ \\
\textbf{Acknowledgements.} The author would like to express his sincere thanks to Wojciech Gajda for suggesting this problem and for many stimulating discussions. The author also thanks Bartosz Naskr\k{e}cki for valuable comments. The author was supported by NCN research grant UMO-2014/15/B/ST1/00128, the Scholarship of Jan Kulczyk for Graduate Students and by the doctoral scholarship of Adam Mickiewicz University.

\section{Preliminaries}\label{s:preliminaries}
Let $R$ be a commutative unital ring with trivial Picard group (e.g., a local or a finite ring), satisfying $6 \in R^{\times}$. Any elliptic curve over $R$ (as defined in~\cite{KatzMazur}) is isomorphic to a projective scheme of the form:
\begin{equation}\label{e:weierstrass_scheme}
  E_{A, B} := \Proj(R[x,y,z]/(y^2 z - x^3 - Ax z^2 - B z^3)),
\end{equation}
for some $A, B \in R$, satisfying: $\Delta(E_{A, B}) := -16 \cdot (4A^3 + 27 B^2) \in R^{\times}$.\\
Moreover, $E_{A, B} \cong E_{a, b}$ if and only if
\begin{equation}\label{e:coefficients}
  A = u^4 \cdot a, \quad B = u^6 \cdot b \qquad \text{ for some } u \in R^{\times}.
\end{equation}
Note that $R$-rational points of $\PP^n$ are in the correspondence with the points of the "naive" projective space:
\begin{equation*}
  \PP^n(R)_{naive} = \{(x_0, \ldots, x_n) \in R^{n+1}: x_0 R + \ldots + x_n R = R\} \bigg/ \sim,
\end{equation*}
where $(x_0, \ldots, x_n) \sim (y_0, \ldots, y_n)$ if and only if $x_i = u \cdot y_i$ for some $u \in R^{\times}$ and all $i$. Therefore $R$-rational points of the Weierstrass curve (\ref{e:weierstrass_scheme}) can be identified with elements of $\PP^2(R)_{naive}$ that satisfy the Weierstrass equation. \\ \\
\noindent
From now on we use the following notation:
\begin{itemize}
  \item $K$ -- a field which is complete with respect to a discrete valuation $v$,
  \item $R$ -- the valuation ring of $v$. It is a local ring with principal maximal ideal $\mf m = (\pi)$,
  \item $k := R / \mf m$ -- the residue field of $v$. We assume that it is perfect and of non-zero characteristic $p \neq 2, 3$,
  \item $R_j := R/\mf m^j$ -- a local ring with maximal ideal $\mf m_j := \mf m/\mf m^j$. For $j = \infty$ we denote: $R_{\infty} = R$.
\end{itemize}
Note that if $i \le j$, then we can reduce any elliptic curve $E/R_j$ to $E_{R_i} := E \times_{R_j} R_i$ over $R_i$. In this way we obtain an exact sequence:
\begin{equation}\label{e:exact_sequence_of_reduction}
  0 \to \wh E (\mf m_j^i) \to E(R_j) \to E_{R_i}(R_i) \to 0
\end{equation}
(surjecivity of the reduction follows from Hensel lemma), where $\wh E$ is a one\nobreakdash-parameter
formal group over $R_j$. By the general theory of formal groups we have the following isomorphism for $i > \frac{e}{p-1}$:
\begin{equation}\label{e:formal_groups}
  \wh E(\mf m_j^i) \cong \mf m_j^i,
\end{equation}
that commutes with the reduction. Note that proofs of \cite[Theorem IV.6.4.]{Silverman2009} and \cite[Proposition VII.2.2.]{Silverman2009} remain valid in this case.

It turns out that elliptic curves over $R_j$ are essentially classified by their $j$\nobreakdash-invariant and the isomorphism class of reduction to $k$. The problem occurs for elliptic curves satisfying $j(E_k) \in \{0, 1728\}$. To treat this issue we introduce the following definition:
  \begin{Definition}
  \textbf{The type} of an elliptic curve $E_{A, B}/R_j$ is $(m, n)$, if
    \begin{equation*}
     A = \pi^m \cdot \alpha, \quad B = \pi^n \cdot \beta \quad \text{ for } \alpha, \beta \in R_j^{\times}
    \end{equation*}
    and $(\infty, 0)$ (respectively $(0, \infty)$), if $A = 0$ (respectively if $B = 0$).\\
    For a curve $E_{A, B}$ of type $(m, 0)$ (where $1 \le m < \infty$) we define:
  \begin{equation*}
   j_1(E_{A, B}) := \frac{B^2}{\alpha^3} \pmod{\mf m_j^{j-m}} \quad \in R_{j - m}.
  \end{equation*}
   Analogously we define $j_1(E_{A, B})$ to be $A^3/\beta^2$ for a curve of $(0, n)$-type.
  \end{Definition}
The condition (\ref{e:coefficients}) assures that the type of a curve, the $j$-invariant given by the usual formula and the invariant $j_1$ do not depend on the choice of the Weierstrass equation.
\begin{Lemma}\label{l:type_classification}
  \begin{enumerate}[(1)]
    \item[]
    \item An isomorphism class of a $(0,0)$-type elliptic curve $E/R_j$ is uniquely determined by its reduction $E_k$ and by the j-invariant $j(E) \in R_j$.
    \item Isomorphism class of an elliptic curve $E/R_j$ of type $(\infty, 0)$ or $(0, \infty)$ is determined uniquely by the isomorphism class of $E_k$.
    \item Assume that $3|(p {-} 1)$. Then the isomorphism class of a curve $E/R_j$ of type $(m, 0)$
     (where $1 \le m < \infty$) is determined uniquely by the isomorphism class of $E_k$ and  $j_1(E) \in R_{j-m}$.
    \item Assume that $4|(p {-} 1)$. Then the isomorphism class of a curve $E/R_j$ of type $(0, n)$ (where $1 \le n < \infty$) is determined uniquely by the isomorphism class of $E_k$ and  $j_1(E) \in R_{j-n}$. 
  \end{enumerate}
\end{Lemma}
\begin{proof}
  (1) Let us assume that elliptic curves $E_{A, B}/R_j$ and $E_{a, b}/R_j$ satisfy:
  \begin{align}
    E_{A, B} \times_{R_j} k &\cong E_{a, b} \times_{R_j} k, \nonumber\\
    j(E_{A, B}) &= j(E_{a, b}), \label{e:j_inv}
  \end{align}
  where $A, B, a, b \in R_j^{\times}$. Then by (\ref{e:coefficients}) there exists $u \in k^{\times}$ such that
  \begin{equation*}
    A \equiv u^4 \cdot a \pmod{\mf m_j}, \qquad B \equiv u^6 \cdot b \pmod{\mf m_j}.
  \end{equation*}
  Using (\ref{e:j_inv}) we obtain:
  \begin{equation}\label{e:00type}
    A^3 \cdot b^2 = a^3 \cdot B^2.
  \end{equation}
  By Hensel lemma the equations:
  \begin{equation*}
    x^4 - A \cdot a^{-1} = 0 \qquad \text{ and } \qquad  x^6 - B \cdot b^{-1} = 0
  \end{equation*}
  have in $R_j$ unique solutions which lift $u$. Let us denote them by $u_1, u_2$ respectively.  
  On the other hand the equality (\ref{e:00type}) implies that both $u_1$ and $u_2$ satisfy the equation:
  \begin{equation*}
    0 = x^{12} - (A \cdot a^{-1})^3 = x^{12} - (B \cdot b^{-1})^2.
  \end{equation*}
  Using again Hensel lemma we see that the above equation has a unique solution in $R_j$. Thus $u_1 = u_2$ and the map $(x, y) \mapsto (u_1^2 \cdot x, u_1^3 \cdot y)$ provides an isomorphism between $E_1$ and $E_2$.\\ \\
  \noindent
  (2), (3), (4) are proven in a similar way. The condition $3|(p {-} 1)$ implies that $\mu_3 \subset k^{\times}$, which allows us to "twist" a lift of $u$ by a suitable cube root of unity in $R_j$. Analogously, if $4|(p-1)$ then $\mu_4 \subset k^{\times}$.
\end{proof}
\begin{Remark}\label{r:E0b_Ea0}
  Note that a curve $E_{0, b}/k$ ($E_{a, 0}/k$ respectively) is ordinary if and only if $3|(p-1)$ ($4|(p-1)$ respectively). Other cases will not be relevant for purposes we have in mind.
\end{Remark}
Let us consider an elliptic curve $\mathbb{E}$ over the residue field $k$. By a theorem of Serre-Tate (\cite[Theorem 1.2.1]{Katz}) lifts of $\mathbb{E}$ to $R_j$ for $j < \infty$ are determined by lifts 
of the $p$-divisible group $\mathbb E_k[p^{\infty}] = (\mathbb E_k[p^n])_n$ to $R_j$. Using the previous classification of elliptic curves we prove an effective version of Serre-Tate theorem:
\begin{Theorem}\label{t:effective_st}
  Let $E_1$, $E_2$ be lifts of an ordinary elliptic curve $\mathbb{E}/k$ to $R_j$. If 
  $E_1[p^{j-1}]$ and $E_2[p^{j-1}]$ are isomorphic as $R_j$-group schemes then $E_1 \cong E_2$.
\end{Theorem}
In order to prove this theorem, we'll have to switch to the algebraic closure $\ol k$. Let $\wh{K^{un}}$ be the completion of the maximal unramified extension $K^{un}/K$, with the ring of integral elements $A$ and the maximal ideal $\mf M = (\pi)$. We denote $A_j := A/\mf M^j$, $\mf M_j := \mf M/\mf M^j$, following our earlier notation. 
\newpage
\begin{Lemma}\label{l:A_j}
  $(A_j^{\times})^{p^m} = (A_j^{\times})^{p^{m+1}}$ for $m \ge j-1$.
\end{Lemma}
\begin{proof}
  We use induction on $j$. For $j=1$ it is immediate. If $a = b^{p^m} \in (A_{j+1}^{\times})^{p^m}$ and $m \ge j$, then by induction hypothesis $b^{p^{m-1}} = c^{p^m} + d$ for some $d \in \mf M^j_{j+1}$. By raising this equality to the $p$-th power and using Newton binomial theorem, we get $a = c^{p^{m+1}}$.
\end{proof}
\begin{proof}[Proof of Theorem \ref{t:effective_st}]
  By the Serre-Tate theorem lifts of an elliptic curve $\mathbb{E}/\ol k$ to $A_j$ are classified by the lifts of its $p$-divisible group. However, since the residue field of $A_j$ is algebraically closed, the \'{e}tale-connected sequence (cf. \cite[p. 43]{Shatz}) of any $p$-divisible group $G$ over the ring $A_j$ lifting $\mathbb{E}[p^{\infty}]$ must be of the form:
  \begin{equation*}
    0 \to \mu_{p^{\infty}} \to G \to \uu{\QQ_p/\ZZ_p} \to 0.
  \end{equation*}
  Thus the lifts of $\mathbb{E}[p^{\infty}]$ to $A_j$ are classified by:
  \begin{equation*}
    \Ext^1_{p-\textbf{div}/A_j}(\uu{\QQ_p/\ZZ_p}, \mu_{p^{\infty}}) = \lim_{\longleftarrow} \Ext^1_{\textbf{GS}/A_j}(\uu{\ZZ/p^n}, \mu_{p^{n}}),
  \end{equation*}
  where $p-\textbf{div}/A_j$ and $\textbf{GS}/A_j$ denote respectively categories of $p$-divisible groups and of group schemes over $A_j$. The Kummer sequence for flat cohomology (cf. \cite[example II.2.18., p. 66]{Milne}) gives us an isomorphism
  \begin{equation*}
    \Ext^1_{\textbf{GS}/A_j}(\uu{\ZZ/p^n}, \mu_{p^{n}}) \cong H^1_{fl}(A_j, \mu_{p^n}) \cong A_j^{\times}/(A_j^{\times})^{p^n}.
  \end{equation*}
  Finally, by Lemma \ref{l:A_j} the natural projection
  \begin{equation*}
   \Ext^1_{p-\textbf{div}/A_j}(\uu{\QQ_p/\ZZ_p}, \mu_{p^{\infty}}) \to \Ext^1_{\textbf{GS}/A_j}(\uu{\ZZ/p^{j-1}}, \mu_{p^{j-1}})
  \end{equation*}
  is an isomorphism. Hence, the lifts of $\mathbb{E}/k$ to $A_j$ are classified by the lifts of $\mathbb{E}[p^{j-1}]$. However,  Lemma \ref{l:type_classification} assures us that if $E_1, E_2/R_j$ are isomorphic over $k$ and $A_j$, then they must be also isomorphic over $R_j$. This ends the proof.
\end{proof}

\section{Proofs of Theorems \ref{t:dp_depends_on_reduction} and \ref{t:implications}}\label{s:low_p_degree}
We assume that $K/\QQ_p$ is a finite extension and stick to the notation from the previous section.
Let also $n = [K:\QQ_p]$, $d = [k:\FF_p]$ and $e$ be the ramification degree of $K/\QQ_p$. For an abelian group $M$ we define $\rank_p M := \dim_{\FF_p} M[p]$. The following lemma computes $\rank_p E_{R_j}(R_j)$ for $j$ big enough and provides a lower bound in the remaining cases. It is a generalization of \cite[Lemma 3.1]{WestonDavid}, which computed $\rank_p E_{R_2}(R_2)$ in the unramified case.
\begin{Lemma}\label{l:main_lemma}
  If $E/\QQ_p$ has good reduction and $i > \frac{e}{p-1}$ is an integer, then:
  \begin{align*}
    \rank_p E_{R_j}(R_j)  &\ge d \cdot (j-i) + \rank_p E(K), & \text{ for } i \le j < i+e,\\
    \rank_p E_{R_j}(R_j)  &= n + \rank_p E(K), & \text{ for } j \ge i+e.
  \end{align*}     
\end{Lemma}
\begin{proof}
  Let $E^{j} := E_{R_j}$ and $E^{i} := E_{R_i}$. By applying the snake lemma to the sequence (\ref{e:exact_sequence_of_reduction}) with multiplication-by-$p$ morphism and using (\ref{e:formal_groups}) we obtain the following commutative diagram:

  \begin{center}
    \begin{tikzpicture}[node distance=3cm, baseline=(current  bounding  box.center)]
      \node (k0) {$0$};
      \node (k1) [right = 1cm of k0] {$0$};
      \node (k2) [right of =  k1] {$E(K)[p]$};
      \node (k3) [right of =  k2] {$E^{i}(R_i)[p]$};
      \node (k4) [right of =  k3] {$\mf m^i/\mf m^{i+e}$};
      \node (c0) [below = 0.6 cm of k0] {$0$};
      \node (c1) [right = 0.65 cm of c0] {$\mf m_j^i[p]$};	
      \node (c2) [right of =  c1] {$E^{j}(R_j)[p]$};
      \node (c3) [right of =  c2] {$E^{i}(R_i)[p]$}; 
      \node (c4) [right of =  c3] {$\mf m_j^i/\mf m_j^{i+e}$};
      \draw[->,font=\scriptsize]
	(k1) edge (c1)
	(k2) edge (c2)
	(k4) edge (c4)
	(k0) edge (k1)
	(k1) edge (k2)
	(k2) edge (k3)
	(k3) edge (k4)
	(c0) edge (c1)
	(c1) edge (c2)
	(c2) edge (c3)
	(c3) edge (c4);
	\draw[double equal sign distance] (k3)--(c3);
    \end{tikzpicture} 
  \end{center} \mbox{} \\ \noindent
  with exact rows. The lower row of the diagram induces an exact sequence:
    \begin{align}\label{e:Ej_sequence}
     0 \to \mf m_j^i[p] \to E^{j}(R_j)[p] \to \ker \left(E^{i}(R_i)[p] \to \mf m_j^i/ \mf m_j^{i+e} \right) \to 0.
    \end{align}\noindent
  Elementary computations show that:
  \begin{align}\label{e:m[p]}
    \mf m_j^i[p] \cong \left\{
    \begin{array}{cc}
      (\ZZ/p)^{d \cdot (j-i)} & \text{ for } j<i+e\\
      (\ZZ/p)^n & \text{ for } j \ge i+e
    \end{array} \right.
  \end{align}
  and that for $j \ge i+e$ the natural projection
  \begin{equation} \label{e:reduction_is_iso}
    \mf m^i/\mf m^{i+e} \to \mf m_j^i/\mf m_j^{i+e}
  \end{equation}
  is an isomorphism. We consider the following two cases:
  \begin{enumerate}[leftmargin=*]

    \item if $i \le j < i+e$, then by (\ref{e:Ej_sequence}) and (\ref{e:m[p]}):
    \begin{align*}
	 \rank_p E^{j}(R_j) &= d \cdot (j-i) + \rank_p \ker(E^{i}(R_i)[p] \to \mf m_j^i/\mf m_j^{i+e})  \\
	 &\ge d \cdot (j-i) + \rank_p E(K)[p].
    \end{align*}

    \item if $j \ge i+e$, then by (\ref{e:reduction_is_iso}):
    \begin{align*}
      \ker(E^{i}(R_i) \to \mf m_j^i/\mf m_j^{i+e}) &= \ker(E^{i}(R_i) \to \mf m^i/\mf m^{i+e}) \cong E(K)[p].
    \end{align*}
    Thus by (\ref{e:Ej_sequence}) the sequence
    \begin{equation*}
      0 \to \mf m_j^i[p] \to E^{j}(R_j)[p] \to E(K)[p] \to 0
    \end{equation*}
    is exact and the proof follows now from (\ref{e:m[p]}). \vspace{-1.7\baselineskip} \qedhere
    \end{enumerate}
\end{proof}
Lemma \ref{l:main_lemma} plays a central role in proofs of Theorems \ref{t:dp_depends_on_reduction} and~\ref{t:implications}. Indeed, the second part of the lemma easily implies Theorem \ref{t:dp_depends_on_reduction}:
\begin{proof}[Proof of Theorem \ref{t:dp_depends_on_reduction}.]
  Let $j = e + \lceil \frac{e}{p-1} \rceil$. Note that $j \le 2e$ and thus we have a natural homomorphism:
  \begin{equation*}
    \ZZ/p^2 \to R_j.
  \end{equation*}
  Therefore $E_{R_j} = (E_{\ZZ/p^2})_{R_j}$ and the proof follows from the second part of Lemma~\ref{l:main_lemma}.
\end{proof}
In order to prove Theorem \ref{t:implications} we investigate the $p$-torsion in fields of low ramification:
\begin{Theorem}\label{t:equivalence} Let $K/\QQ_p$ be a finite extension with ramification degree \\ $e < p - 1$. We keep notation introduced in Section 2. For an elliptic curve $E/\QQ_p$ with good reduction the following conditions are equivalent:
  \begin{enumerate}[(1)]
    \item $E(K)[p] \neq 0$,
    \item $E_{k}(k)[p] \neq 0$ and $\rank_p E_{R_2}(R_2)[p] = d+1$,
    \item $E_{k}(k)[p] \neq 0$ and $E_{\ZZ/p^2}$ is the canonical lift of $E_{\FF_p}$,
    \item $E(K \cap \QQ_p^{un})[p] \neq 0$.
  \end{enumerate}
\end{Theorem}
\begin{proof}
  (1) $\Rightarrow$ (2). By (\ref{e:formal_groups}) it follows that $E(K)[p] \hookrightarrow E_{k}(k)[p]$, so that $E_{k}(k)[p] {\neq} 0$. Therefore the \'{e}tale-connected sequence of $E_{R_2}[p]$ is of the form:
  \begin{equation}\label{e:etale-connected}
    0 \to \mu_p \to E_{R_2}[p] \to \uu{\ZZ/p} \to 0
  \end{equation}
  and thus:
  \begin{equation*}
    \rank_p E_{R_2} (R_2) \le \rank_p \mu_p(R_2) + \rank_p(\uu{\ZZ/p}) = d+1.
  \end{equation*}
  The equality follows from the first part of Lemma \ref{l:main_lemma} applied to $(i, j) {=} (1, 2)$.\\

  (2) $\Rightarrow$ (3). By comparing ranks, we see that the image of $\mu_p(R_2)$ in $E_{R_2}(R_2)[p]$ has index $p$. Let $g \in E_{R_2}(R_2)[p]$ be an element not belonging to the image of $\mu_p(R_2)$. Then $g$ corresponds to a morphism $\uu{\ZZ/p} \to E[p]$, which (after an eventual twist by an automorphism of $\uu{\ZZ/p}$) provides a section of the \'{e}tale-connected sequence (\ref{e:etale-connected}). Thus this sequence splits and $E_{R_2}$ is a canonical lift of $E_k$. This implies that $E_{\ZZ/p^2}$ is a canonical lift of~$E_{\FF_p}$.\\
	  
  (3) $\Rightarrow$ (4). Let us replace $K$ by $K \cap \QQ_p^{un}$. Since $E_{R_2}$ is a canonical lift of~$E_k$:
  \begin{equation*}
    E_{R_2}[p](R_2) \cong \mu_p(R_2) \oplus \uu{\ZZ/p}(R_2) \cong (\ZZ/p)^{d+1}.
  \end{equation*}
  Thus by Lemma \ref{l:main_lemma} $\rank_p E(K) = \rank_p E_{R_2}(R_2) - d = 1$.

  The implication (4) $\Rightarrow$ (1) is obvious.
\end{proof}
\begin{proof}[Proof of Theorem \ref{t:implications}]
  Theorem \ref{t:equivalence} easily implies $(1) \Rightarrow (2) \Leftrightarrow (3)$. The implication $(3) \Rightarrow (4)$ follows from Theorem \ref{t:equivalence} and \cite[Lemma 4.3.]{WestonDavid}.
\end{proof}
\begin{Remark}\label{r:other_implications}
  The condition $(4)$ of Theorem \ref{t:implications} doesn't imply $(3)$ in general. In order to see this consider the curve $E_{1,1}/\QQ_5$. Its division polynomial $\Psi_5$ factors into two polynomials of degrees $2$ and $10$. Therefore $d_5(E) \in \{ 2, 4\}$. Suppose, for a contradiction, that $E(\QQ_5^{un})[5] \neq 0$. Then Theorem \ref{t:equivalence} implies that $E(K)[5] \neq 0$ for some unramified extension $K/\QQ_p$ of degree $\le 4$. On the other hand the degree~$2$ factor of $\Phi_5$ is of the form:
  \begin{equation*}
    x^{2} + \left(2 \cdot 5 + 4 \cdot 5^{2} + \ldots \right) x + \left( 2 \cdot 5^{-1} + 2 + 4 \cdot 5^2 + 5^{3} + \ldots \right).
  \end{equation*}
  One easily checks that roots of the latter polynomial are ramified. This implies that $E(\QQ_5^{un})[5] {=} 0$ and $d_5(E) {=} 4$ (since $d_5(E) {=} 2$ would imply \\ $E(\QQ_5^{un})[5] \neq 0$ by Theorem \ref{t:implications}). Straightforward calculation shows that \\ $\ord_5 a_5(E) = 4$.\\
  
  The implication $(2) \Rightarrow (1)$ is not true either. Indeed, let $E/\QQ_5$ be the canonical lift of $E_{1,1}/\FF_5$. Then $E$ clearly satisfies the condition $(2)$, but $d_5(E) = \ord_5 a_5(E) = 4$. One constructs counterexamples for other primes in a~similar way.\\

\end{Remark}\noindent
The Theorem \ref{t:implications} shows that the elliptic curves with low $p$-degree are well understood. It motivates the following question, which seems to remain open:
\begin{Question}
  What values larger than $(p-1)$ can the $p$-degree attain for a fixed $p$? Are there any divisibility conditions on $d_p(E)$, e.g. $d_p(E)|(p^2 - 1)$?
\end{Question}

\section{Supersingular, CM and multiplicative reduction curves}\label{s:explicit_cases}

In the final section we compute the $p$-degree explicitly in some special cases. We start with elliptic curves with supersingular reduction. Theorem \ref{t:supersingular} will be proven by means of formal groups.
\begin{proof}[Proof of Theorem \ref{t:supersingular}.]
  Note that the multiplication-by-$p$ morphism on $\wh E$ must be of the form:
  \begin{equation*}
    [p](T) = p f(T) + g(T^{p^2}),
  \end{equation*}
  where $f, g \in R \llbracket T \rrbracket$, $g(0) = 0$, $f(T) = T + \ldots$ (cf. Proposition IV.2.3.~(a), Corollary IV.4.4. and Theorem IV.7.4. from \cite{Silverman2009}). Thus, if $E(K)[p] \neq 0$ then for some $x \in \mf m$, $x \neq 0$:
  \begin{equation*}
    0 = p f(x) + g(x^{p^2}),
  \end{equation*}
  which gives:
  \begin{equation*}
    e + v(x) = v(p f(x)) = v(-g(x^{p^2})) \ge v(x^{p^2}) = p^2 \cdot v(x).
  \end{equation*}
  Thus, we get:
  \begin{equation*}
    [K:\QQ_p] \ge e \ge (p^2 - 1) \cdot v(x) \ge p^2 - 1.
  \end{equation*}
\end{proof}
\begin{Corollary}
  Let $E/\QQ$ be an elliptic curve with complex multiplication by $R_K = \ZZ[\omega]$, the maximal order in an imaginary quadratic field $K$. Then for any prime $p \nmid \Delta(E/\QQ)$:
  \begin{align*}
    d_p(E) =
    \begin{cases}
      p^2 - 1, & \text{if $p$ is inert in $R_K$,}\\
      \ord_p(a_p(E)), & \text{if $p$ splits in $R_K$.}
    \end{cases}
  \end{align*}
\end{Corollary}
\begin{proof}
  If $p$ is inert in $R_K$ then by Deuring criterion, $E$ is supersingular at $p$ and the proof follows from Theorem \ref{t:supersingular}. In case if $p$ splits in $R_K$, $E$~is ordinary at~$p$. Moreover, an elliptic curve with complex multiplication is always the canonical lift of its reduction, so it suffices to use Corollary~\ref{t:implications} to finish the argument.
\end{proof}
Theorem \ref{t:explicit_cm} follows now easily by applying the explicit trace formulas.
\begin{proof}[Proof of Theorem \ref{t:explicit_cm}.]
  By \cite[Theorem 18.5]{IrelandRosen}:
  \begin{equation*}
    a_p(E_{D,0}) = \ol{(-D/\pi)_4} \cdot \pi + (-D/\pi)_4 \cdot \ol{\pi},
  \end{equation*}
  where $p = \pi \ol{\pi}$ and $\pi \equiv 1 \pmod{2+2i}$. The formula for $d_p(E_{D, 0})$ follows if we observe that $\pi = s + it$ and
  \begin{equation*}
   (-D/\pi)_4 \equiv (-D)^{(p-1)/4} \pmod{\pi}, \qquad \ol{\pi} \equiv 2s \pmod{\pi}.
  \end{equation*}
  Analogously we get the formula for $d_p(E_{0, D})$, by using \cite[Theorem 18.4]{IrelandRosen}.
\end{proof}
Appying the formula from Theorem \ref{t:explicit_cm} we see that Question \ref{q:asymptotic_behaviour} from the Introduction in the special case of $E_{D, 0}$ comes down to looking for primes in some recurrence sequences.
 
\begin{Corollary}\label{c:recurrence}
  Let $D \in \ZZ$, $D \neq 0$ and let $p \nmid 2D$.

  \begin{enumerate}[(a)]
    \item If $d_p(E_{D, 0}) \in \{1, 2, 4 \}$ then $p = 5$.

    \item $d_p(E_{D, 0}) = 8$ if and only if $p$ is of the form $a_k^2 + a_{k+1}^2$ for some $k \ge 0$, where:
    \begin{equation*}
      a_0 = 0, \qquad  a_1 = 1, \qquad  a_{k+2} = 4 a_{k+1} - a_k.
    \end{equation*}
  \end{enumerate}
\end{Corollary}

\begin{proof}
 Note that $\ord_p \left( (-D)^{\frac{p-1}{4}} \right) |4$. Thus:
    \begin{itemize}
     \item $d_p(E) \in \{ 1, 2, 4 \}$ implies that $\ord_p(2s) \in \{1, 2, 4\}$,
     \item $d_p(E) = 8$ if and only if $\ord_p(2s) = 8$.
    \end{itemize}
 The proof follows now from Theorems 1, 2 and 3 in \cite{CosgraveDilcher}. 
\end{proof}
Finally we investigate the case of curves with multiplicative reduction.
\begin{Theorem}\label{t:multiplicative}
  Let $E = E_{A, B}$ be an elliptic curve over $\QQ_p$ with multiplicative reduction. Then:
  \begin{equation*}
    d_p(E) =
    \begin{cases}
      p-1, & \text{ if } \, j(E) \not \in \QQ_p^p\\
      1, & \text{ if } \, j(E) \in \QQ_p^p, \gamma(E/\QQ_p) \in \QQ_p^2\\
      2, & \text{ if } \, j(E) \in \QQ_p^p, \gamma(E/\QQ_p) \not \in \QQ_p^2,
    \end{cases}
  \end{equation*}
  where $\gamma = -2A/B$ and $\QQ_p^2$, $\QQ_p^p$ are the sets of squares and $p$-th powers in $\QQ_p$ respectively.
\end{Theorem}
\begin{proof}
  Let $L = \QQ_p(\sqrt{\gamma})$ and $x = 1/j(E)$. Then $E/\QQ_p$ is isomorphic over $L$ to the Tate curve $E_q$, where $q$ is given by the power series:
  \begin{equation}\label{e:g_power_series}
    g(x) = x + 744 x^2 + \ldots \in \ZZ[[x]] 
  \end{equation}
  (cf. \cite[Theorem V.5.3]{Silverman1994}). In other words, there is an isomorphism of\\
  $\Gal(\ol{\QQ}_p/L)$\nobreakdash-modules $\Psi : \ol{\QQ}_p^{\times}/\langle q \rangle \to E(\ol{\QQ}_p)$ satisfying:
  \begin{equation}\label{e:twist}
   \Psi \circ \sigma = \chi(\sigma) \cdot \Psi \qquad \text{for every } \sigma \in \Gal(\ol{\QQ}_p/\QQ_p),
  \end{equation}
  where $\chi:\Gal(\ol{\QQ}_p/\QQ_p) \to \{ \pm 1 \}$ denotes the quadratic character associated to $L/\QQ_p$. Observe that for $L = \QQ_p$, the character $\chi$ is trivial. Let $P = \Psi(z) \in E[p]$, $P \neq \mc O$, where:
  \begin{equation*}
   z^p = q^j \qquad \text{for some $j \in \{0, 1, \ldots, p-1\}$.}
  \end{equation*}
  Note that for $\sigma \in \Gal(\ol{\QQ}_p/\QQ_p)$ we have $\sigma \in \Gal(\ol{\QQ}_p/\QQ_p(P))$ if and only if: 
  \begin{equation*}
    \Psi(\sigma(z)) = \sigma(\Psi(z)),  
  \end{equation*}
  which is equivalent by (\ref{e:twist}) to an alternative:
  \begin{equation}
    \sigma \in \Gal(\ol{\QQ}_p/L(z))
    \quad \text{ or } \quad
    \left\{
    \begin{array}{ccc}
      \sigma(z) &\equiv& z^{-1} \pmod{q^{\ZZ}}\\
      \sigma &\not \in& \Gal(\ol{\QQ}_p/L)
    \end{array}
    \right.. \label{e:alternative}
  \end{equation}
  However, $\sigma(z) \equiv z^{-1} \pmod{q^{\ZZ}}$ implies that
  \begin{equation*}
    q^{2j} = (z \cdot \sigma(z))^p \equiv 1 \pmod{q^{p\ZZ}},
  \end{equation*}
  which is possible, if and only if, $j = 0$. Thus, if $j \neq 0$, then we must have $\QQ_p(P) = L(z)$, by the equality of absolute Galois groups. In order to finish the proof we consider two cases:\\

  \textbf{Case I: } Assume that $j \neq 0$. Note that $\QQ_p(z)/\QQ_p$ is totally ramified and $L/\QQ_p$ is unramified, thus $L \cap \QQ_p(z) = \QQ_p$ and:
  \begin{align*}
    [\QQ_p(P):\QQ_p] &= [L(z):\QQ_p] = [\QQ_p(z):\QQ_p] \cdot [L:\QQ_p] = \\
    &=
    \begin{cases}
      1,  & z \in \QQ_p, \, \gamma \in \QQ_p^2,\\
      2,  & z \in \QQ_p, \, \gamma \not \in \QQ_p^2,\\
      \ge p,  & z \not \in \QQ_p.
    \end{cases}
  \end{align*}
  The condition $z \in \QQ_p$ may hold, if and only if, $q \in \QQ_p^p$. Using the equality:
  \begin{equation*}
    \QQ_p^p = \{ p^{pn} \cdot c: \quad n \in \ZZ, \quad  c \in \ZZ_p^{\times}, \quad  c^{p-1} \equiv 1 \pmod{p^2} \},
  \end{equation*}
  which follows by the Hensel lemma, and the formula (\ref{e:g_power_series}), we see that $q \in \QQ_p^p$, if and only if, $j(E) \in \QQ_p^p$.\\
  
  \textbf{Case II:} Assume that $j = 0$, without loss of generality $z = \zeta_p$. In this case the alternative (\ref{e:alternative}) is easily seen to be equivalent to the condition:
  \begin{equation*}
    \sigma(\sqrt{\gamma} \cdot (\zeta_p - \zeta_p^{-1})) = \sqrt{\gamma} \cdot (\zeta_p - \zeta_p^{-1}).
  \end{equation*}
  Thus $\QQ_p(P) = \QQ_p(\sqrt{\gamma} \cdot (\zeta_p - \zeta_p^{-1}))$, which yields that $[\QQ_p(P):\QQ_p] = p-1$.\\
  The proof is complete.
\end{proof}

\address{Graduate School, Faculty of Mathematics and Computer Science\\
Adam Mickiewicz University\\
Umultowska 87, 61-614 Pozna\'{n}, POLAND\\
\email{jgarnek@amu.edu.pl}\\
}

\end{document}